\documentclass[12pt]{article}
\usepackage{mathrsfs}
\usepackage[all]{xy}
\usepackage{amssymb}
\usepackage{amsthm}
\usepackage{amsmath}
\usepackage{enumerate}

\DeclareMathOperator{\Max}{Max}
\DeclareMathOperator{\Min}{Min}

\newcommand{\K}{\mathbf{K}}
\newcommand{\alga}{\mathbf{A}}
\newcommand{\algb}{\mathbf{B}}

\newtheorem{theorem}{Theorem}[section]

\newtheorem{lemma}[theorem]{Lemma}
\newtheorem{proposition}[theorem]{Proposition}

\newtheorem{remark}[theorem]{Remark}
\newtheorem{example}[theorem]{Example}
\newtheorem{corollary}[theorem]{Corollary}
\usepackage[dvipsnames]{xcolor}
\usepackage{graphicx}

\newcommand{\commJP}[1]{{\color{blue} #1}}
\title{Implication in sharply paraorthomodular and relatively paraorthomodular posets}
\author{Ivan~Chajda, Davide~Fazio, Helmut~L\"anger, Antonio~Ledda and \\
Jan Paseka}
\date{}
\begin{document}


\maketitle

\begin{abstract}
In this paper we show that several classes of partially ordered structures having paraorthomodular reducts, or whose sections may be regarded as paraorthomodular posets, admit a quite natural notion of implication, that admits a suitable notion of adjointness. Within this framework, we propose a smooth generalization of celebrated Greechie's theorems on amalgams of finite Boolean algebras to the realm of Kleene lattices.
\end{abstract}



\section{Introduction}
Janusz Czelakowski's early contributions concern quantum logic and the foundation of quantum mechanics. More precisely, he deepened Kochen and Specker's results on partial Boolean algebras (PBAs), which are partial algebras that generalize projection operators over separable Hilbert spaces \cite{KochenSpecker2}. Among other achievements, Czelakowski's work provided, on the one hand, a novel characterization of PBAs embeddable into Boolean algebras. On the other hand, it clarified why PBAs can be regarded as the algebraic semantic of a logic in its own right. See e.g.\ \cite{Czela1975a, Czela1975, Czela1979} for details.\\
The interest in partial Boolean algebras is motivated by their coping quite well with indeterminacy arising when non-commuting observables (self-adjoint linear operators) are considered. In fact, partial Boolean algebras are endowed with a binary operation $\lor$ yielding the algebraic counterpart of a partial disjunction which is defined for pairs of commensurable observables. Therefore, although rooted in a common concrete model (i.e.\ projections over Hilbert spaces), Kochen and Specker's approach to the foundation of quantum mechanics differs from Birkhoff and von Neumann's quantum logic \cite{BirkhoffvonNeumann}. In fact, the former avoids interpretation difficulties stemmed by the latter, which assumes  that conjunctions/disjunctions of quantum propositions always exist, even if they concern incompatible observables.\\
Quoting Czelakowski and Suppes,
\begin{quotation}
``They (\emph{Birkhoff and von Neumann}) require that the structure of experimental propositions be a lattice, and thus that the conjunction of two meaningful propositions in the lattice also be in lattice, but.their requirement seems too strict when one
proposition expresses a possible result of measuring position and the other measuring momentum. The fact that the closed linear subspaces of a complex separable Hilbert space form such a lattice is not sufficient, in my view, to maintain that it expresses the more restricted logic of experimental propositions, in spite of the central importance of this lattice of subspaces in the formulation of the theory''. 
\end{quotation}

Indeed, as mentioned in \cite{Czela1975a, Suppes},

\begin{quotation}
 ``What experimental meaning can one attach to the meet and join of two given experimental propositions".
\end{quotation}

Among classes of partial Boolean algebras, a prominent role is played by \emph{transitive} partial Boolean algebras (see e.g.\ \cite{Czela1975a}) which, in turn, are orthomodular posets whose compatibility relation is \emph{regular} (see e.g.\ \cite{B, K, Pulman} for details). 

Let us recall that a bounded poset $\mathbf{P}=(P,\leq,{}',0,1)$ with a complementation which is an antitone involution is called orthomodular if it is orthogonal, i.e.\ $x\vee y$ exists for $x,y\in P$ whenever $x\leq y'$, and it satisfies the orthomodular law
\begin{enumerate}
\item[(OM)] $x\leq y\ \text{implies}\ x\vee(y\wedge x')=y$.
\end{enumerate}
Moreover, if $\mathbf{P}$ is lattice-ordered, i.e.\ it is an \emph{orthomodular lattice}, then (OM) can be expressed in the form of an identity 
\begin{enumerate}
\item[(OI${}_{\vee}$)] $x\vee\big((x\vee y)\wedge x'\big)\approx x\vee y$
\end{enumerate}
and its dual form 
\begin{enumerate}
\item[(OI${}_{\wedge}$)] $x\wedge\big((x\wedge y)\vee x'\big)\approx x\wedge y$.
\end{enumerate}
A different abstract counterpart of the set of quantum events has been suggested within the so-called unsharp approach to quantum theory. Such structures (i.e.\ effect algebras or MV-algebras) represent a suitable  algebraic counterpart of quantum effects, namely bounded self-adjoint linear operators on a Hilbert space that satisfy the Born's rule. However, since the canonical order induced by Born's rule on the set of effects on a Hilbert space is not a lattice in general, effect algebras do not lend themselves to a smooth logico-algebraic investigation. 

For this reason, in 2016 Roberto Giuntini, Francesco Paoli and one of the present authors proposed the notion of paraorthomodular lattice as natural generalization of the concept of orthomodular lattice. These structures find a concrete realization in the set of effects equipped whit Olson's spectral ordering, see e.g.\ \cite{GLP16} or \cite{GLP17}. However, although paraorthomodular lattices allow to overcome difficulties arising from effect algebras framework in a very natural and physically motivated fashion, they may as well rise interesting philosophical questions, paralleling, in some sense, what happens to their ``sharp relatives''.\\
In \cite{CFLLP} paraorthomodular posets were introduced as a framework in which the order--negation reduct of quantum structures ranging from orthomodular lattices, effect algebras and, of course, paraorthomodular lattices, can be treated into a common backdrop. In the same work, a preliminary investigation of sharply paraorthomodular posets was provided. These structures play for paraorthomodular lattices the role that orthomodular posets possess within the theory of orthomodular lattices. Therefore, the natural question arises whether paraorthomodular posets may be regarded as the semantic counterpart of a logic in its own right.\\

In this paper we address the issue of showing that several classes of partially ordered structures having paraorthomodular reducts, or whose ``sections'' (in the sense of section \ref{sec:relativelyparaorthomodular} below) may be regarded as paraorthomodular posets, admit a quite natural notion of implication. \\
In particular, this paper is organized as follows: after dispatching some preliminary notions in section \ref{sec: base}, in section \ref{sec: impl} we discuss a possible smooth notion of implication in sharply paraorthomodular posets. In section \ref{sec: amalgam}, we provide a generalization of celebrated R.~Greechie's theorems \cite{Greechie1} on amalgams of finite Boolean algebras to the realm of Kleene lattices. In sections \ref{sec:relativelyparaorthomodular} and \ref{sec: rel}, we discuss how the theory of paraortomodular posets carries over in two different degrees of freedom: \emph{relatively} paraorthomodular posets and \emph{relatively} paraorthomodular join-semilattices. Finally, in section \ref{sec: adj}, we address the question if there exists a binary connective that forms an adjoint pair with the implication formerly introduced.
\label{sec: base}\section{Basic concepts}

Let $\mathbf P=(P,\leq)$ be a poset, $A,B\subseteq P$ and $a,b\in P$. We define $A\leq B$ if and only if $x\leq y$ for all $x\in A$ and all $y\in B$. Instead of $A\leq\{b\}$, $\{a\}\leq B$ and $\{a\}\leq\{b\}$ we simply write $A\leq b$, $a\leq B$ and $a\leq b$, respectively. The sets
\begin{align*}
L(A) & :=\{x\in P\mid x\leq A\}, \\
U(A) & :=\{x\in P\mid A\leq x\}
\end{align*}
are called the {\em lower cone} and {\em upper cone} of $A$, respectively. Instead of $L(A\cup B)$, $L(A\cup\{b\})$, $L(\{a,b\})$ and $L\big(U(A)\big)$ we simply write $L(A,B)$, $L(A,b)$, $L(a,b)$ and $LU(A)$, respectively. Analogously, we proceed in similar cases. On $2^P$ we introduce four binary relations 
$\leq$, $\leq_1$, $\leq_2$ and $\approx_2$ as follows:
\begin{align*}
	A\leq B & \text{ if and only if for every }x\in A\text{ and for every }y\in B\text{ we have }x\leq y, \\
   A\leq_1B & \text{ if and only if for every }x\in A\text{ there exists some }y\in B\text{ with }x\leq y, \\
   A\leq_2B & \text{ if and only if for every }y\in B\text{ there exists some }x\in A\text{ with }x\leq y, \\
A\approx_2B & \text{ if and only if }A\leq_2B\text{ and }B\leq_2A.
\end{align*}

Recall that $A\leq y$ if and only if $A\leq_1 y$, and similarly $x\leq B$ 
 if and only if $x\leq_2B$. 

A unary operation ${}'$ on $P$ is called an {\em antitone involution} if
\[
x,y\in P\text{ and }x\leq y\text{ together imply }y'\leq x'
\]
and if it satisfies the identity
\[
x''\approx x.
\]
Note that if $\mathbf P$ has an antitone involution $'$ and a bottom element $0$, then it has also a top element $1$, namely $0'$. We will call $x,y\in P$ {\em orthogonal} to each other, shortly $x\perp y$, if $x\leq y'$ or, equivalently, $y\leq x'$.

By an {\em orthogonal poset} we mean a bounded poset $(P,\leq,{}',0,1)$ with an antitone involution satisfying the following condition:
\[
x,y\in P\text{ and }x\perp y\text{ together imply that }x\vee y\text{ exists}.
\]
Here and in the following $x\vee y$ denotes the supremum of $x$ and $y$.

The concept of a paraorthomodular poset was introduced in \cite{CFLLP} as a generalization of the concept of a paraorthomodular lattice introduced in \cite{GLP17}. Recall that a bounded {\em poset} $\mathbf P=(P,\leq,{}',0,1)$ with an antitone involution is called {\em paraorthomodular} if it satisfies the following condition:
\begin{enumerate}
\item[(P)] $x,y\in P$, $x\leq y$ and $x'\wedge y=0$ together imply $x=y$.
\end{enumerate}
It should be remarked that $x'\wedge y=0$ is equivalent to $L(x',y)=0$. The second notation was used in our previous paper \cite{CFLLP}. Moreover, if the paraorthomodular poset $\mathbf P$ is lattice-ordered then it is called a \emph{paraorthomodular lattice}.

Given a poset $\mathbf P$, an element $y \in P$ is said to be a {\em complement} of $x \in P$ if $L(x, y) =L(P)$ and $U(x, y) = U(P)$. $\mathbf P$ is said to be {\em complemented} if each element of $P$ has a complement in $\mathbf P$. Moreover, if $\mathbf P$ is a bounded poset with antitone involution which is orthogonal and ${}'$ is a complementation, then $\mathbf P$ is an orthomodular poset. Namely, as shown e.g.\ in \cite K, for a complementation ${}'$, (P) equivalent to (OM).

An orthogonal paraorthomodular poset is called {\em sharply paraorthomodular} (see \cite{CFLLP}). Of course, any paraorthomodular lattice is sharply paraorthomodular. We call a sharply paraorthomodular poset $\mathbf P$ \emph{regular} provided it satisfies the following identity, for any $x,y\in P$:
\begin{enumerate}\label{reg}
\item[(reg)] $x\land x'\leq y\lor y'$.\footnote{It is worth observing that paraorthomodular lattices defined in \cite{GLP16} coincide with regular paraorthomodular lattices in our framework.}
\end{enumerate}
Note that, by virtue of orthogonality, \eqref{reg} is well-defined. We call a distributive, regular paraorthomodular lattice a \emph{Kleene lattice}.

A {poset} $\mathbf P$ is called {\em distributive} if one of the following equivalent LU-identities is satisfied:
\begin{align*}
	L\big(U(x,y),z\big) & \approx LU\big(L(x,z),L(y,z)\big), \\
	U\big(L(x,z),L(y,z)\big) & \approx UL\big(U(x,y),z\big), \\
	U\big(L(x,y),z\big) & \approx UL\big(U(x,z),U(y,z)\big), \\
	L\big(U(x,z),U(y,z)\big) & \approx LU\big(L(x,y),z\big), \\
	L\big(U(x_1,x_2, \dots,x_n),z\big) & \approx LU\big(L(x_1,z),L(x_2,z), \dots, L(x_n,z)\big),\\
	U\big(L(x_1,x_2, \dots,x_n),z\big) & \approx UL\big(U(x_1,z),U(x_2,z), \dots, U(x_n,z)\big).
\end{align*}

Recall that a {\em Boolean poset} is a bounded distributive complemented poset.
Since the complementation in a Boolean poset $\mathbf P$ is unique and antitone
$\mathbf P$ is a poset with antitone involution.

An orthomodular poset $\mathbf P$ is called {\em weakly Boolean} \cite{tkadlec} if for 
every $a, b \in P$ the condition 
$a\wedge b=a\wedge b'=0$ implies $ a= 0$.

Let $\mathbf P=(P,\leq)$ be a poset. In what follows, for every subset $A$ of $P$ let $\Max A$ and $\Min A$ denote the set of all maximal and minimal elements of $A$, respectively. We say that $\mathbf P$ is
\begin{enumerate}
	\item {\em mub-complete} \cite{abramsky} if for every upper bound $x$
	of a finite  subset $M$ of $P$ there is a minimal upper bound of $M$ below $x$, 
	\item {\em mlb-complete}  if for every lower bound $x$
	of a finite subset $M$ of $P$ there is a maximal lower bound of $M$ above $x$, 
	\item {\em mlub-complete}  if it is both mub-complete and mlb-complete. 
\end{enumerate}

If $\mathbf P=(P,\leq,{}',0,1)$ is a bounded poset with an antitone involution then 
$\mathbf P$ is mub-complete if and only if it is mlb-complete. 

A poset $\mathbf P$ is said to satisfy the {\em Ascending Chain Condition} (ACC) if it has no infinite ascending chains. The notion of {\em Descending Chain Condition} (DCC) is defined dually. If $\mathbf P$ satisfies the ACC then every non-empty subset of $P$ has at least one maximal element. The dual statement holds for the DCC.

Evidently, if $\mathbf P$ satisfies the ACC then $\mathbf P$ is mlb-complete and the dual statement holds for the DCC. Moreover, every finite poset and every lattice is mlub-complete. 

The bounded poset $\mathbf P$ has the {\em maximality property} \cite{tkadlec} if
for every $a, b \in P$ the set $L(a,b)$ has a maximal element. Clearly, 
every mlb-complete poset has the property of maximality. 
Recall also that every weakly Boolean orthomodular poset with the property
of maximality is a Boolean algebra \cite[Theorem 4.2]{tkadlec}.

We sometimes extend an operator $*\colon P^2\rightarrow2^P$ to a binary operation on $2^P$ by
\[
A*B:=\bigcup_{x\in A,y\in B}(x*y)
\]
for all $A,B\in2^P$.

\label{sec: impl}\section{What implication for sharply paraorthomodular posets?}

Let $(P,\leq,{}',0,1)$ be an orthogonal mlb-complete poset. 
If $y\in P$ and $A\subseteq P$ then we will denote the set $\{y\vee a\mid a\in A\}$ by $y\vee A$ provided that all the elements $y\vee a$ exist. Analogously, we proceed in similar cases. Now we introduce the operator $\rightarrow\colon P^2\rightarrow2^P$ as follows:
\begin{enumerate}
\item[(I1)] $x\rightarrow y:=y\vee\Max L(x',y')$.
\end{enumerate}
Since $\Max L(x',y')\leq y'$, $x\rightarrow y$ is correctly defined.

It is worth noticing that every {\em orthomodular poset} (see e.g.\ \cite{CL21a}, \cite{CL22}, \cite{FLP} and \cite F) is paraorthomodular, but there are examples of paraorthomodular posets (or even lattices, indeed) which are not orthomodular, see \cite{GLP17}.

\begin{example}
The paraorthomodular posets depicted in Fig.~1 are neither orthomodular nor sharply paraorthomodular:\\

\vspace*{-4mm}

\begin{center}
\setlength{\unitlength}{7mm}
\begin{picture}(22,8)
\put(3,1){\circle*{.3}}
\put(1,3){\circle*{.3}}
\put(5,3){\circle*{.3}}
\put(1,5){\circle*{.3}}
\put(5,5){\circle*{.3}}
\put(3,7){\circle*{.3}}
\put(3,1){\line(-1,1)2}
\put(3,1){\line(1,1)2}
\put(3,7){\line(-1,-1)2}
\put(3,7){\line(1,-1)2}
\put(1,3){\line(0,1)2}
\put(1,3){\line(2,1)4}
\put(5,3){\line(-2,1)4}
\put(5,3){\line(0,1)2}
\put(2.85,.3){$0$}
\put(.35,2.85){$a$}
\put(5.4,2.85){$b$}
\put(.35,4.85){$b'$}
\put(5.4,4.85){$a'$}
\put(2.3,7.4){$1=0'$}
\put(2.6,-.75){{\rm(a)}}
\put(11,1){\circle*{.3}}
\put(8,3){\circle*{.3}}
\put(10,3){\circle*{.3}}
\put(12,3){\circle*{.3}}
\put(14,3){\circle*{.3}}
\put(8,5){\circle*{.3}}
\put(10,5){\circle*{.3}}
\put(12,5){\circle*{.3}}
\put(14,5){\circle*{.3}}
\put(11,7){\circle*{.3}}
\put(11,1){\line(-3,2)3}
\put(11,1){\line(-1,2)1}
\put(11,1){\line(1,2)1}
\put(11,1){\line(3,2)3}
\put(11,7){\line(-3,-2)3}
\put(11,7){\line(-1,-2)1}
\put(11,7){\line(1,-2)1}
\put(11,7){\line(3,-2)3}
\put(8,3){\line(0,1)2}
\put(8,3){\line(1,1)2}
\put(8,3){\line(2,1)4}
\put(10,3){\line(-1,1)2}
\put(10,3){\line(2,1)4}
\put(12,3){\line(-2,1)4}
\put(12,3){\line(1,1)2}
\put(14,3){\line(-2,1)4}
\put(14,3){\line(-1,1)2}
\put(14,3){\line(0,1)2}
\put(10.85,.3){$0$}
\put(7.35,2.85){$a$}
\put(9.35,2.85){$b$}
\put(12.4,2.85){$c$}
\put(14.4,2.85){$d$}
\put(7.35,4.85){$d'$}
\put(9.35,4.85){$c'$}
\put(12.4,4.85){$b'$}
\put(14.4,4.85){$a'$}
\put(10.3,7.4){$1=0'$}
\put(10.6,-.75){{\rm(b)}}
\put(10.25,-1.55){{\rm Fig.~1}}
\put(8,-2.35){paraorthomodular posets}
\put(19,1){\circle*{.3}}
\put(17,3){\circle*{.3}}
\put(19,3){\circle*{.3}}
\put(21,3){\circle*{.3}}
\put(17,5){\circle*{.3}}
\put(19,5){\circle*{.3}}
\put(21,5){\circle*{.3}}
\put(19,7){\circle*{.3}}
\put(19,1){\line(-1,1)2}
\put(19,1){\line(0,1)6}
\put(19,1){\line(1,1)2}
\put(19,7){\line(-1,-1)2}
\put(19,7){\line(1,-1)2}
\put(17,3){\line(0,1)2}
\put(17,3){\line(1,1)2}
\put(17,3){\line(2,1)4}
\put(19,3){\line(-1,1)2}
\put(19,3){\line(1,1)2}
\put(21,3){\line(-2,1)4}
\put(21,3){\line(-1,1)2}
\put(21,3){\line(0,1)2}
\put(18.85,.3){$0$}
\put(16.35,2.85){$a$}
\put(19.4,2.85){$b$}
\put(21.4,2.85){$c$}
\put(16.35,4.85){$c'$}
\put(19.4,4.85){$b'$}
\put(21.4,4.85){$a'$}
\put(18.3,7.4){$1=0'$}
\put(18.6,-.75){{\rm(c)}}
\end{picture}
\end{center}

\vspace*{15mm}

The posets visualized in Fig.~2 are sharply paraorthomodular:

\vspace*{-4mm}

\begin{center}
\setlength{\unitlength}{7mm}
\begin{picture}(15,8)
\put(3,1){\circle*{.3}}
\put(1,3){\circle*{.3}}
\put(5,3){\circle*{.3}}
\put(1,5){\circle*{.3}}
\put(5,5){\circle*{.3}}
\put(3,7){\circle*{.3}}
\put(3,1){\line(-1,1)2}
\put(3,1){\line(1,1)2}
\put(3,7){\line(-1,-1)2}
\put(3,7){\line(1,-1)2}
\put(1,3){\line(0,1)2}
\put(5,3){\line(0,1)2}
\put(2.85,.3){$0$}
\put(.35,2.85){$a$}
\put(5.4,2.85){$b$}
\put(.35,4.85){$a'$}
\put(5.4,4.85){$b'$}
\put(2.3,7.4){$1=0'$}
\put(2.6,-.75){{\rm(a)}}
\put(11,1){\circle*{.3}}
\put(8,3){\circle*{.3}}
\put(10,3){\circle*{.3}}
\put(12,3){\circle*{.3}}
\put(14,3){\circle*{.3}}
\put(8,5){\circle*{.3}}
\put(10,5){\circle*{.3}}
\put(12,5){\circle*{.3}}
\put(14,5){\circle*{.3}}
\put(11,7){\circle*{.3}}
\put(11,1){\line(-3,2)3}
\put(11,1){\line(-1,2)1}
\put(11,1){\line(1,2)1}
\put(11,1){\line(3,2)3}
\put(11,7){\line(-3,-2)3}
\put(11,7){\line(-1,-2)1}
\put(11,7){\line(1,-2)1}
\put(11,7){\line(3,-2)3}
\put(8,3){\line(1,1)2}
\put(8,3){\line(2,1)4}
\put(10,3){\line(-1,1)2}
\put(10,3){\line(1,1)2}
\put(10,3){\line(2,1)4}
\put(12,3){\line(-2,1)4}
\put(12,3){\line(-1,1)2}
\put(12,3){\line(1,1)2}
\put(14,3){\line(-2,1)4}
\put(14,3){\line(-1,1)2}
\put(10.85,.3){$0$}
\put(7.35,2.85){$a$}
\put(9.35,2.85){$b$}
\put(12.4,2.85){$c$}
\put(14.4,2.85){$d$}
\put(7.35,4.85){$d'$}
\put(9.35,4.85){$c'$}
\put(12.4,4.85){$b'$}
\put(14.4,4.85){$a'$}
\put(10.3,7.4){$1=0'$}
\put(10.6,-.75){{\rm(b)}}
\put(6.25,-1.55){{\rm Fig.~2}}
\put(2.8,-2.35){sharply paraorthomodular posets}
\end{picture}
\end{center}

\vspace*{12mm}

The first one is a lattice, the second one is not. Moreover, the second one is not orthomodular since $b\leq d'$, but $b\vee(d'\wedge b')=b\vee b=b\neq d'$.
\end{example}

We first list several elementary properties of the operator $\rightarrow$ defined by (I1) in orthogonal posets.

\begin{theorem}\label{th1}
Let $(P,\leq,{}',0,1)$ be an orthogonal mlb-complete poset, 
$\rightarrow$ defined by {\rm(I1)} and $x,y\in P$. Then the following hold:
\begin{enumerate}[{\rm(i)}]
\item $y\leq x\rightarrow y$,
\item $x\leq y$ implies $y\rightarrow z\leq_1x\rightarrow z$,
\item $x\rightarrow y\approx\left\{
\begin{array}{ll}
y\vee y'           & \text{ if }x\leq y, \\
y\vee(x'\wedge y') & \text{ if }x\perp y, \\
x'\vee y           & \text{ if }y\leq x,
\end{array}
\right.$
\item $(x\rightarrow y)\rightarrow y\approx y\vee\big(y'\wedge\Min U(x,y)\big)$,
\item $\big((x\rightarrow y)\rightarrow y\big)\rightarrow y\approx y\vee\Big(y'\wedge\big(y\vee\Max L(x',y')\big)\Big)$.
\end{enumerate}
\end{theorem}

Let us note that if ${}'$ is a complementation (e.g., if $\mathbf P$ is orthomodular) then the first line of (iii) can be written in the form
\begin{enumerate}
\item[(iii')] $x\leq y$ implies $x\rightarrow y=1$.
\end{enumerate}
Implication in orthomodular lattices is usually defined by $x\rightarrow y:=y\vee(x'\wedge y')$.

\begin{proof}[Proof of Theorem~\ref{th1}]
\begin{enumerate}[(i)]
\item follows directly from (I1).
\item Assume $x\leq y$ and $z\in P$. Then $y'\leq x'$. Let $u\in y\rightarrow z$. Then there exists some $w\in\Max L(y',z')$ with $z\vee w=u$. Since $y'\leq x'$ we have $w\in L(x',z')$. Thus there exists some $t\in\Max L(x',z')$ with $w\leq t$. Now
\[
u=z\vee w\leq z\vee t\in x\rightarrow z
\]
proving $y\rightarrow z\leq_1x\rightarrow z$.
\item $x\rightarrow y\approx y\vee\Max L(x',y')\approx\left\{
\begin{array}{ll}
y\vee y' & \text{ if }x\leq y, \\
y\vee(x'\wedge y') & \text{ if }x\perp y, \\
y\vee x'=x'\vee y & \text{ if }y\leq x.
\end{array}
\right.$
\item
\begin{align*}
(x\rightarrow y)\rightarrow y & \approx\bigcup_{z\in x\rightarrow y}(z\rightarrow y)\approx\bigcup_{z\in y\vee\Max L(x',y')}\big(y\vee\Max L(z',y')\big)\approx \\
                              & \approx\bigcup_{w\in\Max L(x',y')}\big(y\vee\Max L(y'\wedge w',y')\big)\approx\bigcup_{w\in\Max L(x',y')}\big(y\vee(y'\wedge w')\big)\approx \\
                              & \approx y\vee\big(y'\wedge\Min U(x,y)\big).
\end{align*}
\item According to (iv) and (iii) we have
\begin{align*}
\big((x\rightarrow y)\rightarrow y\big)\rightarrow y & \approx\Big(y\vee\big(y'\wedge\Min U(x,y)\big)\Big)\rightarrow y\approx \\
                                                     & \approx\bigcup_{w\in\Min U(x,y)}\Big(\big(y\vee(y'\wedge w)\big)\rightarrow y\Big)\approx \\
                                                     & \approx\bigcup_{w\in\Min U(x,y)}\Big(\big(y'\wedge(y\vee w')\big)\vee y\Big)\approx \\
																										 & \approx y\vee\Big(y'\wedge(y\vee\Max L(x',y')\big)\Big).
\end{align*}
\end{enumerate}
\end{proof}

If $\mathbf P$ is, moreover, sharply paraorthomodular, we can prove a bit more.

\begin{lemma}
Let $(P,\leq,{}',0,1)$ be a sharply paraorthomodular mlb-complete poset, 
$\rightarrow$ defined by {\rm(I1)} and $a,b\in P$. Then the following hold:
\begin{enumerate}[{\rm(i)}]
\item $b'\rightarrow b=b$.
\item If $a\perp b$ and $b\wedge b'=0$ then $a\rightarrow b=b$ if and only if $a=b'$.
\item $a\rightarrow b=1$ implies $a\leq b$.
\end{enumerate}
\end{lemma}

\begin{proof}
\begin{enumerate}[(i)]
\item 	$b'\rightarrow b=b\vee(b\wedge b')=b$.
\item Assume $a\perp b$ and $b\wedge b'=0$. First suppose $a\rightarrow b=b$. Then $b\vee\Max L(a',b')=a\rightarrow b=b$. Hence $\Max L(a',b')\leq b$. This shows $\Max L(a',b')\leq b\wedge b'=0$, i.e.\ $\Max L(a',b')=0$ and therefore $L(a',b')=0$, i.e.\ $a'\wedge b'=0$ which together with $a\leq b'$ and (P) yields $a=b'$. The other implication follows from (i). 
\item Assume $a\rightarrow b=1$. Let $c\in\Max L(a',b')$. Then $c\leq b'$ and $b\vee c=1$, i.e.\ $c'\wedge b'=0$. Applying (P) we conclude $c=b'$. Hence
\[
b'=c\in\Max L(a',b')\subseteq L(a',b')\subseteq L(a')
\]
and thus $b'\leq a'$, i.e.\ $a\leq b$.
\end{enumerate}
\end{proof}

As shown above, the operator $\rightarrow$ in sharply paraorthomodular posets satisfies several nice properties and hence it can be considered as the connective {\em implication} in the logic formalized by means of paraorthomodular posets (see \cite{CFLLP} for the connection of this logic with the logic of quantum mechanics, see also \cite{FLP} and \cite{GLP17} for the lattice version). Hence the natural question arises whether this logic can be based on this connective only. Contrary to the case of orthomodular posets treated from this point of view in \cite{CL22}, this is not possible since we cannot derive the order $\leq$ by means of the implication $\rightarrow$ only. However, we can get a condition on an orthogonal poset $\mathbf P$ equipped with $\rightarrow$ formulated in this language such that $\mathbf P$ becomes a sharply paraorthomodular one.

\begin{theorem}
Let $\mathbf P=(P,\leq,{}',0,1)$ be an orthogonal mlb-complete poset 
and $\rightarrow$ defined by {\rm(I1)}. Then the following are equivalent:
\begin{enumerate}[{\rm(i)}]
\item $\mathbf P$ is paraorthomodular.
\item If $x\rightarrow y=1$ then $x\leq y$.
\end{enumerate}
\end{theorem}

\begin{proof}
(i) $\Rightarrow$ (ii): \\
Assume that $x\rightarrow y=1$. Let $z\in\Max L(x',y')$. Then 
$y\vee z=1$, i.e., $y'\wedge z'=0$.  Since $z\leq y'$ we have according to (P) that 
$z= y'$, i.e., $y'\leq x'$. We conclude  $x\leq y$. \\
(ii) $\Rightarrow$ (i): \\
If $x\leq y$ and $x'\wedge y=0$ then 
\[
y\rightarrow x=y'\vee x=(x'\wedge y)'=0'=1
\]
which implies $y\leq x$ by (ii) and together with $x\leq y$ yields $x=y$.
\end{proof}

In the case of lattices, (I1) reduces to
\begin{enumerate}
\item[(I2)] $x\rightarrow y:=y\vee(x'\wedge y')$.
\end{enumerate}
The next lemma shows that $\rightarrow$ defined by (I2) is antitone in 
the first variable provided the poset in question is a lattice.

\begin{lemma}
Let $(L,\vee,\wedge,{}')$ be a lattice with an antitone mapping ${}'$ and $\rightarrow$ defined by {\rm(I2)}. Then $x\leq y$ implies $y\rightarrow z\leq x\rightarrow z$.
\end{lemma}

\begin{proof}
$x\leq y$ implies $y'\leq x'$ and hence $y\rightarrow z=z\vee(y'\wedge z')\leq z\vee(x'\wedge z')=x\rightarrow z$.
\end{proof}

\section{Amalgams of Kleene lattices}\label{sec: amalgam}
In this section we provide a generalization of celebrated R.~Greechie's theorems \cite{Greechie1} on amalgams of finite Boolean algebras to the realm of Kleene lattices. Specifically, we will show how to obtain sharply paraorthomodular posets and paraorthomodular lattices by ``gluing'' together Kleene lattices pairwise sharing a common atomic subalgebra of at most four elements. Interestingly enough, we will show that \emph{any} amalgam of Kleene lattices $\mathscr{K}$ yields a paraorthomodular poset $\K$. Such a structure is also a sharply paraorthomodular poset (paraorthomodular lattice) provided that its building Kleene blocks do not form loops of order 3 (loops of order 3 or 4). Analogous results have been already obtained for (lattice) effect algebras in terms of pastings of MV algebras, see e.g.\ \cite{Chovanec,xie}.

Let $\mathscr{K}=\{\K_{i}\}_{i\in I}$ be a family of finite Kleene lattices such that $|K_{i}|\geq 6$, for any $i\in I$. Moreover, let $K_{i}\cap K_{j}$ $(i\ne j)$ be either $\{0,1\}$ or an atomic Kleene subalgebra of both $K_{i}$ and $K_{j}$ (on which operations coincide) such that $|K_{i}\cap K_{j}|\leq 4$. Also, we assume that any $x\in K_{i}\cap K_{j}\smallsetminus\{0,1\}$ is either an atom or a co-atom in both $K_{i}$ and $K_{j}$. We call $\mathscr{K}$ a \emph{pasted family} of Kleene lattices, and any $\K_{i}\in\mathscr{K}$ an \emph{initial block}. Note that $\K_{i}\cap\K_{j}\not\cong\K_{3}$, where $\K_{3}$ is the three-elements Kleene lattice, since otherwise one would have $\K_{i}\cong\K_{3}\cong\K_{j}$. In fact, assume by way of contradiction that $\K_{i}\cap\K_{j}\cong\K_{3}$; let $a\in K_{i}\smallsetminus K_{3}$. Clearly one must have $a\parallel b$, where $b$ is the fixed point of $\K_{3}$. By regularity, one must have $0=a\land a'<b$,  and so $\K_{i}$ contains a sublattice isomorphic to $\mathbf{N}_{5}$, which is impossible. Therefore, $\K_{i}\cap\K_{j}$ must be orthoisomorphic either to the two-element Boolean algebra $\algb_{2}$, or to the four-element Boolean algebra $\algb_{4}$, or to the four-element Kleene lattice: 
\vspace*{-4mm}

\begin{center}
\setlength{\unitlength}{7mm}
\begin{picture}(6,8)
\put(3,1){\circle*{.3}}
\put(3,3){\circle*{.3}}
\put(3,5){\circle*{.3}}
\put(3,7){\circle*{.3}}
\put(3,1){\line(0,1)2}
\put(3,3){\line(0,1)2}
\put(3,5){\line(0,1)2}
\put(2.85,.3){$0$}
\put(3.4,2.85){$x$}
\put(3.4,4.85){$x'$}
\put(2.3,7.4){$1=0'$}
\put(2.2,-.7){{\rm Fig.~3}}
\end{picture}
\end{center}

\vspace*{3mm}

Let $\mathscr{K}=\{\K_{i}\}_{i\in I}$ be a pasted family of Kleene lattices. Set ${K}=\bigcup_{i\in I}K_{i}$, and let $\leq\ \subseteq K^{2}$ be such that, for any $x,y\in K$, \[x\leq y\text{ if and only if there exists }i\in I\text{ such that }x\leq^{\K_{i}}y.\] Customary arguments yield the following
\begin{lemma}\label{lem: partordamal}Let $\mathscr{K}=\{\K_{i}\}_{i\in I}$ be a pasted family of Kleene lattices. Then $(K,\leq,0,1)$ is a bounded partially ordered set.
\end{lemma}
The following lemma is obvious
\begin{lemma}\label{lem:antinvamalgams}Let $\mathscr{K}=\{\K_{i}\}_{i\in I}$ be a pasted family of Kleene lattices. Upon defining, for any $x\in K$, $x'=x'^{\K_{i}}$, for some $i\in I$ such that $x\in K_{i}$, one has that $(K,\leq,{}',0,1)$ is a bounded poset with antitone involution. 
\end{lemma}
Given a pasted family of Kleene lattices $\mathscr{K}=\{\K_{i}\}_{i\in I}$, we will call the bounded poset with antitone involution $(K,\leq,{}',0,1)$ obtained from a pasted family of Kleene lattices $\mathscr{K}=\{\K_{i}\}_{i\in I}$ as by Lemma \ref{lem: partordamal} and Lemma \ref{lem:antinvamalgams} an \emph{atomic amalgam} of Kleene lattices. Note that, if any $\K_{i}\in\mathscr{K}$ is a Boolean algebra (for any $i\in I$), then our concept coincides with the one developed by R.~Greechie (cf.\ \cite[p.~144]{B}).
The following result shows that, indeed, any atomic amalgam of Kleene lattices yields a paraorthomodular poset.
\begin{lemma}\label{lem: antitoninvolutpartord}Let $\mathscr{K}$ be a pasted family of Kleene lattices, then $\K=(K,\leq,{}',0,1)$ is a paraorthomodular poset. 
\end{lemma}

\begin{proof}
Of course, $(K,\leq,{}',0,1)$ is a bounded poset with antitone involution by Lemma \ref{lem: partordamal} and Lemma \ref{lem: antitoninvolutpartord}. Now, suppose towards a contradiction that $\K$ is not paraorthomodular. Then, by \cite[Theorem 2.5]{CFLLP}, $\K$ contains a strong sub-poset orthoisomorphic to the Benzene ring $\algb_{6}$. 
\vspace*{-4mm}

\begin{center}
\setlength{\unitlength}{7mm}
\begin{picture}(6,8)
\put(3,1){\circle*{.3}}
\put(1,3){\circle*{.3}}
\put(5,3){\circle*{.3}}
\put(1,5){\circle*{.3}}
\put(5,5){\circle*{.3}}
\put(3,7){\circle*{.3}}
\put(3,1){\line(-1,1)2}
\put(3,1){\line(1,1)2}
\put(3,7){\line(-1,-1)2}
\put(3,7){\line(1,-1)2}
\put(1,3){\line(0,1)2}
\put(5,3){\line(0,1)2}
\put(2.85,.3){$0$}
\put(.35,2.85){$y'$}
\put(5.4,2.85){$x$}
\put(.35,4.85){$x'$}
\put(5.4,4.85){$y$}
\put(2.3,7.4){$1=0'$}
\put(2.2,-.7){{\rm Fig.~4}}
\end{picture}
\end{center}

\vspace*{3mm}

Since $x\leq^{\K_{i}}y$, for some $i\in I$, one has that $\K_{i}$ contains a sublattice isomorphic to $\commJP{\mathbf{N}_{5}}$, which is impossible. 
\end{proof}
\begin{remark}Let $\mathscr{K}=\{\K_{i}\}_{i\in I}$ be a pasted family of Boolean algebras and $J\subseteq I$. Then $(\bigcup_{j\in J}K_{j},\leq,',0,1)$ is a sub-poset with antitone involution of $\K$. 
\end{remark}
\begin{lemma}\label{lem: uniontwoblockpoml}Let $K=\bigcup_{i\in I}K_{i}$ be an atomic amalgam of a pasted family of Kleene lattices $\{\K_{i}\}_{i\in I}$. Then, for any $i,j\in I$ such that $i\ne j$, $K_{i}\cup K_{j}$ is the base set of a paraorthomodular lattice which is isomorphic to the amalgam pasting together the lattices $(K_{i},\land,\lor)$ and $(K_{j},\land,\lor)$ along the pasted sublattice $(K_{i}\cap K_{j},\land,\lor)$.
\end{lemma}

\begin{proof}
Note that, since $K_{i}\cap K_{j}$ is finite, it is a closed plexus in the sense of \cite[p.139]{B}. Moreover, $K_{i}\cup K_{j}$ is a regularly pasted amalgam of lattices which, by the first part of \cite[Theorem 3.4]{B} is the universe of a lattice upon defining greatest lower bounds and least upper bounds as in the proof of \cite[Lemma 3.3]{B}. Moreover, by Lemma \ref{lem: antitoninvolutpartord}, $(K_{i}\cup K_{j},\land,\lor,{}',0,1)$ is a paraorthomodular lattice.
\end{proof}

Let $\alga=(A,\leq)$ be a poset. For any $x,y\in A$, we say that $x$ \emph{covers} $y$, written $x\lessdot y$, provided that, for any $z\in A$ one has that $x\leq z\leq y$ implies $x=z$ or $y=z$.

\begin{lemma}\label{lem: cover}Let $\mathscr{K}=\{\K_{i}\}_{i\in I}$ be a pasted family of Kleene lattices and let $x,y\in K$ be such that $x\leq^{\K_{i}}y$ for some $i\in I$. Then the following holds:
	\begin{enumerate}[{\rm(i)}]
		\item If $x\lessdot y$ in $\K$ then $x\lessdot^{\K_{i}}y$ in $\K_{i}$.
		\item If $y\ne x'$  then 
		$x\lessdot y$ in $\K$ if and only if $x\lessdot^{\K_{i}}y$  in $\K_{i}$.
	\end{enumerate}
\end{lemma}

\begin{proof}
\begin{enumerate}[(i)]
\item If $x\lessdot y$ in $\K$ and $x<^{\K_{i}}z<^{\K_{i}}y$ then $x<z<y$ would hold also in $\K$, a contradiction. 
\item Let us suppose that $x\lessdot^{\K_{i}}y$ but $x<z<y$. Therefore, for some $j,k\in I$, one has $x<^{\K_{j}}z<^{\K_{k}}y$. We have $x\in K_{i}\cap K_{j}$, $z\in K_{j}\cap K_{k}$, and $y\in K_{i}\cap K_{k}$. Observe also that $i\ne j$ and $i\ne k$. Let us distinguish several cases depending on the nature of $x$ and $y$. If $x=0$, then $y$ is an atom in $\K_{i}$ and hence also in $\K_{k}$.  Similarly, if $y=1$, then $x$ is a co-atom in $\K_{i}$ and in $\K_{j}$. This means that $z=x$ or $z=y$, a contradiction.	Therefore, let us suppose without loss of generality that $x\ne 0$ and $y\ne 0$. Hence $x$ is an atom in both $\K_{i}$ and $\K_{j}$ and $y$ is a co-atom in both $\K_{i}$ and $\K_{k}$. Now, if $j\ne k$, then $z$ must be either an atom in $\K_{k}$ and hence in $\K_{j}$, and in this case $z=x$, 
	or a co-atom  in $\K_{j}$ and hence in $\K_{k}$, and so $z=y$, which is impossible. So we must have $j=k$. But this means that $\K_{i}\cap\K_{j}=\{0, x, y=x',1\}$, contradicting our assumptions.
\end{enumerate}
\end{proof}

\begin{remark}
Observe that, when one deals with the general case of Kleene lattices, the assumption of Lemma \ref{lem: cover}{\rm(ii)} is necessary. Indeed, let us consider the amalgam of Kleene lattices $\K$
\vspace*{-4mm}

\begin{center}
\setlength{\unitlength}{7mm}
\begin{picture}(6,8)
\put(4,0){\circle*{.3}}
\put(0,4){\circle*{.3}}
\put(8,4){\circle*{.3}}
\put(4,2){\circle*{.3}}
\put(2,4){\circle*{.3}}
\put(6,4){\circle*{.3}}
\put(4,6){\circle*{.3}}
\put(4,8){\circle*{.3}}
\put(4,0){\line(0,1)2}
\put(4,2){\line(-1,1)2}
\put(4,2){\line(-2,1)4}
\put(8,4){\line(-2,1)4}
\put(4,2){\line(1,1)2}
\put(0,4){\line(1,1)4}
\put(4,0){\line(1,1)4}
\put(2,4){\line(1,1)2}
\put(6,4){\line(-1,1)2}
\put(4,6){\line(0,1)2}
\put(4.32,-.3){$0$}
\put(-0.6,3.8){$b$}
\put(8.3,3.8){$b'$}
\put(4.3,1.8){$a$}
\put(1.4,3.8){$c$}
\put(6.3,3.8){$c'$}
\put(4.3,6){$a'$}
\put(4.3,8){$1$}
\put(3.2,-1.2){{\rm Fig.~5}}
\end{picture}
\end{center}

\vspace*{7mm}

Note that $a\lessdot a'$ in the isomorphic copy of $\K_{3}\times\algb_{2}$ $\{0,1,a,a',b,b'\}$, but $a<c<a'$ in $\K$.
\end{remark}
Borrowing the analogous notion for Boolean algebras, given a natural number $n\geq 3$, we say that initial blocks $\K_{i_{1}},\dots,\K_{i_{n}}$ of an atomic amalgam $\K$ form an \emph{atomic loop} of order $n$ for $\K$ provided that the following conditions are satisfied:
\begin{enumerate}
\item For any $1\leq j< n$, $|K_{i_{j}}\cap K_{i_{j+1}}|=4$. Moreover, $|K_{i_{n}}\cap K_{i_{1}}|=4$.
\item For any other pair of indexes $1\leq j,k\leq n$ not mentioned in the above item, $K_{i_{j}}\cap K_{i_{k}}=\{0,1\}$.
\item  For any $1\leq j<k<l<n$, $K_{i_{j}}\cap K_{i_{k}}\cap K_{i_{l}}=\{0,1\}$.
\end{enumerate}
Note that item $(3)$ ensures that ,for any $1\leq j<n$, $\{0,1\}=K_{i_{j}}\cap K_{i_{j+1}}\cap K_{i_{j+2}}=\{0,1,a_{j}\commJP{,}a'_{j}\}\cap \{0,1,a_{j+1}\commJP{,}a'_{j+1}\}$. Reasoning in a similar way one has that the atoms $a_{1},\dots,a_{n}$ must be pairwise distinct. 
\begin{lemma}\label{lem: existence join no loop 3}Let $\mathscr{K}=\{\K_{i}\}_{i\in I}$ be a pasted family of Kleene lattices which does not contain an atomic loop of order $3$. If $x,y\in K_{i}$, then $x\lor^{\K_{i}}y$ is the least upper bound of $x$ and $y$ in $\K$.
\end{lemma}

\begin{proof}
Let us assume without loss of generality that $x\parallel y$. Let $z\in K$ be such that $x,y\leq z$. 
If $z=1$ then $x\lor^{\K_{i}}y\leq z$. Suppose now that $z\ne 1$. 
	Then there exists $j,k\in I$ such that $x<^{\K_{j}}z$ and $y <^{\K_{k}}z$, 
	$x$ is an atom in  both $\K_{j}$ and $\K_{i}$, $y$ is an atom in  both $\K_{k}$ and $\K_{i}$, and 
	$z$ is a co-atom in  both $\K_{j}$ and $\K_{k}$. If $j,i,k$ are pairwise distinct, then $(\K_{i},\K_{j},\K_{k})$ forms an atomic loop of order $3$. Therefore, one must have $i=j$ or $j=k$. Therefore, by Lemma \ref{lem: uniontwoblockpoml}, $K_{i}\cup K_{j}\cup K_{k}$ 
	forms a paraorthomodular lattice and so we have again $x\lor^{\K_{i}}y\leq z$.
\end{proof}

\begin{remark}\label{rem: kleenehalbo}
Recall that in any Kleene lattice $\K$ it holds that \[x\land y=0\text{ implies }x\leq y'.\]
\end{remark}

\begin{theorem}\label{thm: greechie1}
An atomic amalgam of Kleene lattices is a sharply paraorthomodular poset if and only if it does not contain any loop of order $3$.
\end{theorem}

\begin{proof}
Suppose that the atomic amalgam $\mathbf K$ of Kleene lattices does not contain an atomic loop of order $3$. For any $x,y\in K$, if $x\leq y'$, then there exists $i\in I$ such that $x,y\in K_{i}$. But then $x\lor y$ exists and equals $x\lor^{\K_{i}}y$, by Lemma \ref{lem: existence join no loop 3}. Moreover, $\K$ is a paraorthomodular poset by Lemma \ref{lem: antitoninvolutpartord} and so it is sharply paraorthomodular. Conversely, let us assume that $\K$ contains an atomic loop of order $3$, say $(\K_{1},\K_{2},\K_{3})$. Assume that $K_{1}\cap K_{2}=\{0,1,a_{1},a'_{1}\}$, $K_{2}\cap K_{3}=\{0,1,a_{2},a'_{2}\}$, $K_{3}\cap K_{1}=\{0,1,a_{3},a'_{3}\}$, where $a_{1},a_{2},a_{3}$ are pairwise distinct atoms. We show that $a_{1}\lor a_{3}$ does not exist. Let us assume by way of contradiction that this supremum exists. Since $a_{1}\land a_{2}=0=a_{2}\land a_{3}$, by Remark \ref{rem: kleenehalbo}, one must have $a_{1}\lor a_{3}\leq a'_{2}$. \\
Now, let us suppose that $a_{1}\lor^{\K_{1}} a_{3}$ equals $a'_{1}$ or $a'_{3}$. Let us assume without loss of generality that $a_{1}\lor^{\K_{1}} a_{3}=a'_{1}$. Then one has $a_{1}\leq^{\K_{i}} a_{1}\lor a_{3}\leq^{\K_{j}}a'_{1}$. If $i=1$ or $j=1$, then one must have $a_{1}\lor a_{3}=a'_{1}$ and so $a'_{1}\leq a'_{2}$, i.e.\ $a_{1}=a_{2}$, which is absurd. So we must have $i\ne 1$ and $j\ne 1$. But then $a_{1}\lor a_{3}$ must be either an atom, or a co-atom, which is equally impossible. Therefore, we can assume without loss of generality that $a_{1}\lor^{\K_{1}} a_{3}$ is not equal to $a'_{1}$ or $a'_{3}$. Note that $a_{1}\lessdot^{\K_{1}}a_{1}\lor^{\K_{1}}a_{3}$. In fact, if $a_{1}<x<a_{1}\lor^{\K_{1}}a_{3}$, then the following would be a sub-lattice of $\K_{1}$:
\vspace*{-4mm}

\begin{center}
\setlength{\unitlength}{7mm}
\begin{picture}(6,8)
\put(3,1){\circle*{.3}}
\put(1,3){\circle*{.3}}
\put(5,4){\circle*{.3}}
\put(1,5){\circle*{.3}}
\put(3,7){\circle*{.3}}
\put(3,1){\line(-1,1)2}
\put(3,1){\line(2,3)2}
\put(3,7){\line(-1,-1)2}
\put(1,3){\line(0,1)2}
\put(5,4){\line(-2,3)2}
\put(2.85,.3){$0$}
\put(.3,2.85){$a_{1}$}
\put(.35,4.85){$x$}
\put(5.3,3.8){$a_{3}$}
\put(2.0,7.4){$a_{1}\lor^{\K_{1}}a_{3}$}
\put(2.2,-.7){{\rm Fig.~6}}
\end{picture}
\end{center}

\vspace*{3mm}

But this contradicts the distributivity of $\K_{1}$. Similarly, $a_{3}\lessdot^{\K_{1}}a_{1}\lor^{\K_{1}}a_{3}$. So, by Lemma \ref{lem: cover}, we must have that $a_{1}, a_{3}\lessdot a_{1}\lor^{\K_{1}}a_{3}$ in $\K$. Hence, we have $a_{1}\lor^{\K_{1}}a_{3}=a_{1}\lor a_{3}\leq^{\K_{j}} a'_{2}$, for some $j\in I$. If $j=1$, then $a'_{2}\in K_{1}\cap K_{2}\cap K_{3}=\{0,1\}$, but this is impossible since $a'_{2}$ is a co-atom. If $i\ne 1$, then  $a_{1}\lor^{\K_{1}}a_{3}\ne a'_{2}$ and it must be a co-atom but this is impossible, since one would have $a'_{2}=1$, which is absurd.
\end{proof}
\begin{theorem}An atomic amalgam of Kleene lattices is a paraorthomodular lattice if and only if it does not contain an atomic loop of order $3$ or $4$.
\end{theorem}

\begin{proof}
If $\K$ contains a loop of order $3$ then it is not a lattice by Theorem \ref{thm: greechie1}. Therefore, we can assume without loss of generality that $\K$ does not contain loops of order $3$. Suppose that 
$\K$ contains a loop $(\K_{1},\K_{2},\K_{3},\K_{4})$ of order $4$. Let $K_{i}\cap K_{i+1}=\{0,1,a_{i},a'_{i}\}$ and $K_{4}\cap K_{1}=\{0,1,a_{4},a'_{4}\}$, 
$a_1, a_2, a_3, a_4$ atoms. Let us assume towards a contradiction that $a_{1}\lor a_{3}$ exists. Note that
\begin{enumerate}
\item[($\ast$)] $a_{1}\lor a_{3}\leq a'_{2},a'_{4}$.
\end{enumerate}
Therefore, there are $i,j\in I$ such that $a_{1}\lor a_{3}\leq^{\K_{i}}a'_{2}$ and $a_{1}\lor a_{3}\leq^{\K_{j}} a'_{4}$. If $i\ne j$, then $a_{1}\lor a_{3}$ is either an atom, or a co-atom, or $1$. But this is impossible because of ($\ast$). So we conclude that $i=j$. Now, we have $a_{1}\leq^{\K_{k}}a_{1}\lor a_{3}$. Therefore, reasoning as above, we conclude that $k=i$. But then $(\K_{1},\K_{i},\K_{4})$ is a loop of order $3$, which is impossible.\\
Conversely, suppose that $\K$ does not contain loops of order $3$ or $4$. Assume towards a contradiction that $\{x,y\}\subseteq K$ does not have a least upper bound \emph{A fortiori} this means that $x\parallel y$ (and, of course, that $x,y\notin\{0,1\}$). Let us consider the following cases.
\begin{enumerate}
\item There exists $i\in I$ such that $x,y\in K_{i}$. By hypothesis there is $z\in K$ such that $x\lor^{\K_{i}} y\not\leq z$. Clearly, $z\ne 1$. Moreover, we have also $x\leq^{\K_{j}}z$ and $y\leq^{\K_{k}}z$. If $j=i=k$, then $x\lor^{\K_{i}} y\leq z$ which is impossible. If e.g.\ $i\ne j$ and $k=i$ then $y$ must be an atom. So $x\land y = 0$ and, by Remark \ref{rem: kleenehalbo}, $x\leq y'$. But then, by Theorem \ref{thm: greechie1}, $x\lor y$ exists, a contradiction. So one must have $i\ne j$ and $i\ne k$ and so it is easily seen that $(\K_{i},\K_{j},\K_{k})$ forms a loop of order $3$, against our assumptions. 

\item If $x$ and $y$ are in distinct blocks, then let $x\leq^{\K_{i}}z$ and $y\leq^{\K_{j}}z$ with $z\ne 1$ and consider $u\in K$ such that $x\leq^{\K_{k}}u$,  $y\leq^{\K_{l}}u$ but $z\not\leq u$ (and so $u\ne 1$). By hypothesis, one has $i\ne j$. Therefore $z$ must be a co-atom. Now, if $j\ne l$ and $i\ne k$, since necessarily $k\ne l$, one has that $(\K_{i},\K_{j},\K_{l},\K_{k})$ is a loop of order $4$, against our assumptions. Therefore, one must have $j=l$ or $i=k$. In both cases, it can be seen that one must have $u=z$. Again a contradiction.
\end{enumerate}
\end{proof}

\section{Relatively paraorthomodular posets}\label{sec:relativelyparaorthomodular}
In the sequel we introduce posets closely related to paraorthomodular ones.

A {\em relatively paraorthomodular poset} is an ordered quintuple $\big(P,\leq,({}^x;x\in P),0,1\big)$ such that $(P,\leq,0,1)$ is a bounded poset and for every $x\in P$, $([x,1],\leq,{}^x,x,1)$ is a paraorthomodular poset. Hence, a bounded poset $(P,\leq,0,1)$ is relatively paraorthomodular if in each principal filter $[x,1]$ there exists an antitone involution ${}^x$ such that
\begin{enumerate}
\item[(RP)] $x\leq y\leq z$ and $y^x\wedge z=x$ imply $y=z$.
\end{enumerate}

\begin{example}An example of a poset {\rm(}even a lattice{\rm)} where every proper interval $[x,1]$ can be equipped with an antitone involution, but which is not relatively paraorthomodular, is visualized in Fig.~7.

		\begin{tabular}{p{0.50\textwidth} p{0.100\textwidth} p{0.400\textwidth}} 
It is evident how to introduce the involution in every interval $[x,1]$; for example, in $[a,1]$ we have $a^a:=1$, $(b')^a:=b'$ and $1^a:=a$. However this poset is not paraorthomodular and hence also not relatively paraorthomodular since $a\leq b'$ and $a'\wedge b=0$, but $a\neq b'$. If, however, the involution in $[0,1]$ is defined as in Fig.~2 then this poset is relatively paraorthomodular.& &
\begin{minipage}{0.400\textwidth}
	\vskip1.8cm
\setlength{\unitlength}{6mm}
\begin{picture}(6,8)
\put(3,1){\circle*{.3}}
\put(1,3){\circle*{.3}}
\put(5,3){\circle*{.3}}
\put(1,5){\circle*{.3}}
\put(5,5){\circle*{.3}}
\put(3,7){\circle*{.3}}
\put(3,1){\line(-1,1)2}
\put(3,1){\line(1,1)2}
\put(3,7){\line(-1,-1)2}
\put(3,7){\line(1,-1)2}
\put(1,3){\line(0,1)2}
\put(5,3){\line(0,1)2}
\put(2.85,.3){$0$}
\put(.35,2.85){$a$}
\put(5.4,2.85){$b$}
\put(.35,4.85){$b'$}
\put(5.4,4.85){$a'$}
\put(2.3,7.4){$1=0'$}
\put(2.25,-.75){{\rm Fig.~7}}
\put(-.8,-1.55){non-paraorthomodular lattice}
\end{picture}
\end{minipage}
\end{tabular}

\end{example}

\begin{example}
Examples of relatively paraorthomodular posets which are not lattices are visualized in Fig.~1.
\end{example}

It is a question whether also relatively paraorthomodular lattices and posets can be converted into a certain logic, i.e.\ whether we can introduce a connective implication $\rightarrow$ having some expected properties.

For mub-complete relatively paraorthomodular posets we can define
\begin{enumerate}
\item[(I3)] $x\rightarrow y:=\big(\Min U(x,y)\big)^y$
\end{enumerate}
where $A^y:=\{a^y\mid a\in A\}$. Since $y\leq U(x,y)$ we have $\Min U(x,y)\subseteq[y,1]$ and hence $x\rightarrow y$ is well-defined.

Several interesting properties of the operator $\rightarrow$ defined by (I3) in a relatively paraorthomodular poset are listed in the following result.

\begin{theorem}\label{th2}
Let $\mathbf P=\big(P,\leq,({}^z;z\in P),0,1\big)$ be a relatively paraorthomodular 
mub-complete poset, 
$\rightarrow$ defined by {\rm(I3)} and $x,y\in L$. Then the following hold:
\begin{enumerate}[{\rm(i)}]
\item $y\leq x\rightarrow y$,
\item $x\leq y$ if and only if $x\rightarrow y=1$,
\item $x\rightarrow y=\left\{
\begin{array}{ll}
1           & \text{ if }x\leq y, \\
(x\vee y)^y & \text{ if } x\vee y \text{ exists}, \\
x^y         & \text{ if }y\leq x,
\end{array}
\right.$
\item $(x\rightarrow y)\rightarrow y\approx\Min U(x,y)$,
\item $\big((x\rightarrow y)\rightarrow y\big)\rightarrow y\approx x\rightarrow y$.
\end{enumerate}
\end{theorem}

\begin{proof}
\begin{enumerate}[(i)]
\item follows directly from (I3).
\item The following are equivalent: $x\leq y$, $\Min U(x,y)=y$, $\big(\Min U(x,y)\big)^y=1$, $x\rightarrow y=1$.
\item $x\rightarrow y=\big(\Min U(x,y)\big)^y=\left\{
\begin{array}{ll}
y^y=1 & \text{ if }x\leq y, \\
(x\vee y)^y & \text{ if }x\vee y\text{ exists},\\
x^y & \text{ if }y\leq x.
\end{array}
\right.$
\item
\begin{align*} 
(x\rightarrow y)\rightarrow y & \approx\bigcup_{z\in x\rightarrow y}(z\rightarrow y)\approx\bigcup_{z\in\big(\Min U(x,y)\big)^y}\big(\Min U(z,y)\big)^y\approx \\
                              & \approx\bigcup_{w\in\Min U(x,y)}\big(\Min U(w^y,y)\big)^y\approx\bigcup_{w\in\Min U(x,y)}w^{yy}\approx\Min U(x,y).
\end{align*}
\item According to (iv) and (iii) we have
\begin{align*}
\big((x\rightarrow y)\rightarrow y\big)\rightarrow y & \approx\Min U(x,y)\rightarrow y\approx\bigcup_{z\in\Min U(x,y)}(z\rightarrow y)\approx\bigcup_{z\in\Min U(x,y)}z^y\approx \\
                                                     & \approx\big(\Min U(x,y)\big)^y\approx x\rightarrow y.
\end{align*}
\end{enumerate}
\end{proof}

Surprisingly, the implication as defined by (I3) can be used for a characterization of paraorthomodularity, see the next theorem.

\begin{theorem}
Let $(P,\leq,0,1)$ be a bounded mub-complete poset, 
for all $x\in P$ let $([x,1],{}^x,x,1)$ be a bounded poset with an antitone involution and let $\rightarrow$ be defined by {\rm(I3)}. Then the following are equivalent:
\begin{enumerate}[{\rm(i)}]
\item $(P,\leq,{}^0,0,1)$ is paraorthomodular.
\item $x\leq y^0$ and $x\rightarrow y=y$ together imply $x=y^0$.
\end{enumerate}
\end{theorem}

\begin{proof}
The following are equivalent: $x\rightarrow y=y$, $\big(\Min U(x,y)\big)^y=y$, $\Min U(x,y)=1$, $U(x,y)=1$, $x\vee y=1$, $x^0\wedge y^0=0$. Hence (ii) is equivalent to (ii'):
\begin{enumerate}
\item[(ii')] $x\leq y^0$ and $x^0\wedge y^0=0$ imply $x=y^0$.
\end{enumerate}
But this is nothing else than (P) for $(P,\leq,{}^0,0,1)$.
\end{proof}

Our results can be illuminated by the following example.

\begin{example}
Consider the relatively paraorthomodular poset depicted in {\rm Fig.~1(a)}. In sections define antitone involutions as follows: \\
$[a,1]$: $(a')^a:=b'$, $(b')^a:=a'$, $a^a:=1$, $1^a:=a$, \\
$[b,1]$: $(a')^b:=b'$, $(b')^b:=a'$, $b^b:=1$, $1^b:=b$ \\
and trivially in the two-element intervals $[a',1]$ and $[b',1]$. Then the table for the operator $\rightarrow$ defined by $x\rightarrow y:=\big(\Min U(x,y)\big)^y$ is as follows:
\[
\begin{array}{c|cccccc}
\rightarrow & 0  &     a     &     b     & a' & b' & 1 \\
\hline
     0      & 1  &     1     &     1     & 1  & 1  & 1 \\
     a      & a' &     1     & \{a',b'\} & 1  & 1  & 1 \\
     b      & b' & \{a',b'\} &     1     & 1  & 1  & 1 \\
     a'     & a  &     b'    &     b'    & 1  & b' & 1 \\
     b'     & b  &     a'    &     a'    & a' & 1  & 1 \\
     1      & 0  &     a     &     b     & a' & b' & 1
\end{array}
\]
\end{example}

If we assume a certain compatibility condition, we can prove a stronger result.

\begin{theorem}
Let $(P,\leq,0,1)$ be a bounded mub-complete poset 
and for every $x\in P$ let $^x$ be an antitone involution on $([x,1],\leq)$ such that the following compatibility condition holds:
\begin{enumerate}
\item[{\rm(C)}] $x\leq y\leq z$ implies $z^y=z^x\vee y$.
\end{enumerate}
Moreover, let $\rightarrow$ be defined by {\rm(I3)}. Then the following are equivalent:
\begin{enumerate}[{\rm(i)}]
\item $\mathbf P=\big(P,\leq,(^x;x\in P),0,1\big)$ is relatively paraorthomodular.
\item $x\rightarrow y=1$ implies $x\leq y$.
\end{enumerate}
\end{theorem}

\begin{proof}
(i) $\Rightarrow$ (ii): \\
This follows from Theorem~\ref{th2}. \\
(ii) $\Rightarrow$ (i): \\
We want to prove that $([x,1],{}^x,x,1)$ is paraorthomodular. If $x\leq y\leq z$ and $y^x\wedge z=x$
\[
z\rightarrow y=\big(\Min U(z,y)\big)^y=z^y=z^x\vee y=(z\wedge y^x)^x=x^x=1
\]
according to (C) and hence $z\leq y$ according to (ii) which together with $y\leq z$ yields $y=z$.
\end{proof}

\section{Relatively paraorthomodular join-semilattices}\label{sec: rel}

More interesting results can be obtained if we assume our paraorthomodular poset to be a join-semilattice. Then definition (I3) of the implication reduces to the following one:
\begin{enumerate}
\item[(I4)] $x\rightarrow y:=(x\vee y)^y$
\end{enumerate}
where $(x\vee y)^y$ denotes the result of applying the involution $^y$ in the interval $[y,1]$ to $x\vee y$. Since $x\vee y\in[y,1]$, $(x\vee y)^y$ is well-defined.

\begin{example}
Consider the lattice $\mathbf L=(L,\leq,{}',0,1)$ depicted in Fig.~8: 

\vspace*{-4mm}

\begin{center}
\setlength{\unitlength}{7mm}
\begin{picture}(11,8)
\put(4,1){\circle*{.3}}
\put(1,3){\circle*{.3}}
\put(3,3){\circle*{.3}}
\put(5,3){\circle*{.3}}
\put(7,3){\circle*{.3}}
\put(4,5){\circle*{.3}}
\put(6,5){\circle*{.3}}
\put(8,5){\circle*{.3}}
\put(10,5){\circle*{.3}}
\put(7,7){\circle*{.3}}
\put(4,1){\line(-3,2)3}
\put(4,1){\line(-1,2)1}
\put(4,1){\line(1,2)1}
\put(4,1){\line(3,2)6}
\put(1,3){\line(3,2)6}
\put(4,5){\line(-1,-2)1}
\put(4,5){\line(1,-2)1}
\put(4,5){\line(3,-2)3}
\put(7,3){\line(-1,2)1}
\put(7,3){\line(1,2)1}
\put(7,7){\line(-1,-2)1}
\put(7,7){\line(1,-2)1}
\put(7,7){\line(3,-2)3}
\put(3.85,.3){$0$}
\put(.35,2.85){$a$}
\put(2.35,2.85){$b$}
\put(5.4,2.85){$c$}
\put(7.4,2.85){$d$}
\put(3.35,4.85){$d'$}
\put(5.35,4.85){$c'$}
\put(8.4,4.85){$b'$}
\put(10.4,4.85){$a'$}
\put(6.3,7.4){$1=0'$}
\put(3.2,-.75){{\rm Fig.~8}}
\put(-.5,-1.55){relatively paraorthomodular lattice}
\end{picture}
\end{center}

\vspace*{8mm}

Evidently, this poset is paraorthomodular. Consider e.g.\ the interval $[a,1]$. Then we define an involution ${}^a$ as follows: $a^a:=1$, $(d')^a:=d'$, $1^a:=a$; similarly for the intervals $[b,1]$, $[c,1]$ and trivially for two-element intervals $[d',1]$, $[c',1]$, $[b',1]$ and $[a',1]$. For $[d,1]$ we define
\[
d^d:=1, (a')^d:=d', (b')^d:=c', (c')^d:=b', (d')^d:=a'\text{ and, of course, }1^d:=d.
\]
One can easily check that every interval of the form $[x,1]$ is a paraorthomodular poset and hence $\mathbf L$ is a relatively paraorthomodular lattice.
\end{example}

\begin{remark}
Assume that $\mathbf L=(L,\vee,\wedge,{}',0,1)$ is an orthomodular lattice and $a,b\in L$ with $a\leq b$. It is well-known {\rm(}cf.\ e.g.\ {\rm\cite K)} that then $\big([a,b],\vee,\wedge,x\mapsto(x'\vee a)\wedge b,a,b\big)$ is an orthomodular lattice, too. Hence $\mathbf L$ is relatively paraorthomodular and satisfies {\rm(C)} since
\[
x\leq y\leq z\text{ implies }z^y=z'\vee y=z'\vee x\vee y=z^x\vee y
\]
and {\rm(I4)} has the form
\[
x\rightarrow y=(x\vee y)^y=\big(y\vee(x\vee y)'\big)\wedge1=y\vee(x'\wedge y')
\]
since $x\vee y\in[y,1]$. We see that this definition of $x\rightarrow y$ coincides with {\rm(I2)}.
\end{remark}

\begin{lemma}
Let $\big(L,\vee,(^x;x\in L),0,1\big)$ be a relatively paraorthomodular join-semilattice and $\rightarrow$ be defined by {\rm(I4)}. Then $x\leq y$ implies $y\rightarrow z\leq x\rightarrow z$.
\end{lemma}

\begin{proof}
$x\leq y$ implies $x\vee z\leq y\vee z$ and hence $y\rightarrow z=(y\vee z)^z\leq(x\vee z)^z=x\rightarrow z$.
\end{proof}

\section{Adjointness}\label{sec: adj}

Having defined the connective implication $\rightarrow$, e.g., as shown in the previous sections, the question arises if there exists a binary connective $\odot$ such that they form a so-called {\em adjoint pair}, i.e.
\[
x\odot y\leq z\text{ if and only if }x\leq y\rightarrow z.
\]
It is worth noticing that in case that $\odot$ and $\rightarrow$ are binary operations, any of the two operations $\odot$ and $\rightarrow$ uniquely determines the other one. Namely, for given $x$ and $y$, the element $x\odot y$ is the smallest element $z$ satisfying $x\leq y\rightarrow z$, and for given $y$ and $z$, the element $y\rightarrow z$ is the greatest element $x$ satisfying $x\odot y\leq z$.

Let us note that we do not require $\odot$ to be associative or commutative. We will show that the requirement that $\odot$ and $\rightarrow$ form an adjoint pair is not realistic for paraorthomodular lattices or posets and hence we consider only the 
following two pairs of weaker conditions (A) and (B), and (A)${}_{21}$ and (B)${}_{12}$ respectively:
\medskip

\begin{minipage}{0.49\textwidth}
\begin{enumerate}
\item[(A)] $x\odot y\leq z$ implies $x\leq y\rightarrow z$,
\item[(B)] $x\leq y\rightarrow z$ implies $x\odot y\leq z$.
\end{enumerate}
\end{minipage}
\begin{minipage}{0.49\textwidth}
\begin{enumerate}
	\item[(A)${}_{21}$] $x\odot y\leq_2 z$ implies $x\leq_1 y\rightarrow z$,
	\item[(B)${}_{12}$] $x\leq_1 y\rightarrow z$ implies $x\odot y\leq_2 z$.
\end{enumerate}
\end{minipage}

When we will work with posets instead of lattices then the expressions $x\odot y$ and $y\rightarrow z$ can be subsets of the poset in question and hence we sometimes use the relations $\leq_1$ and $\leq_2$ instead of the partial order relation $\leq$.

In what follows, we will consider also another $\rightarrow$, 
the so-called Sasaki implication.

To simplify our reasonings, we will use for an orthogonal mlb-complete poset $\mathbf P$ or 
a lattice $\mathbf L$ with an involution ${}'$ the following 
notation:
\begin{align*}
x\rightarrow_{I} y=y\vee \Max L(x',y'), x\rightarrow_{S} y=x'\vee \Max L(x,y), 
 x\odot_{S} y=y\wedge \Min U(x,y'). 
\end{align*}

Note that in the case when $\mathbf L$ is a lattice with an antitone involution, $x\odot_{S} y$ is the so-called Sasaki projection and 
$ x\rightarrow_{S} y$  is the so-called Sasaki implication. 
Moreover, $x\rightarrow_{I} y=y'\rightarrow_{S} x'$. 

In fact we have the following lemma. 

\begin{lemma}\label{lemadj}
	Let $\mathbf L=(L,\vee,\wedge,{}')$ be a lattice with  an antitone 
	involution ${}'$ on $\mathbf L$. Then {\rm(A)} and {\rm(B)}
	for the connectives $\odot_{S}$  and $ \rightarrow_{S} $ are equivalent, respectively.
\end{lemma}

\begin{proof} Assume that {\rm(A)} holds. Let $a,b,c\in L$ and assume $a\leq b\rightarrow_{S} c$.
	Then $a\leq b'\vee(c\wedge b)$. Put $x=c'$, $y=b$, and $z=a'$. Hence 
	$z'\leq y'\vee(x'\wedge y)$. We obtain 
	$x\odot_{S} y= y\wedge (x\vee y')\leq z$. By {\rm(A)} we have 
	$x\leq y\rightarrow_{S} z=y'\vee(z\wedge y)$, i.e., $c'\leq b'\vee(a'\wedge b)$. 
	We conclude that $a\odot_{S} b=b\wedge (a\vee b')\leq c $ and {\rm(B)} holds. 
	The reverse implication follows analogously. 	
\end{proof}

Moreover, for the Sasaki implication $\rightarrow_S$, we obtain an adjoint pair provided the lattice in question satisfies even weaker conditions, see the following result.


\begin{theorem}\label{omidentity}	Let $\mathbf L=(L,\vee,\wedge,{}')$ be a 
	lattice with  an involution ${}'$. Then the following conditions  are equivalent .
	\begin{enumerate}[{\rm (i)}]
		\item The orthomodular identities {\rm(OI${}_{\vee}$)} and {\rm(OI${}_{\wedge}$)}
				hold in $\mathbf L$. 
		\item $x\odot_S y\leq z$ if and only if $x\leq y\rightarrow_S z$.
	\end{enumerate}	
\end{theorem}

\begin{proof}
(i) $\Rightarrow$ (ii): \\
$a\odot_S b\leq c$ implies
	\begin{align*}
		a & \leq a\vee b'\stackrel{\mbox{\rm(OI${}_{\vee}$)}}{=}b'\vee\big((a\vee b')\wedge b\big)=b'\vee\big(b\wedge(a\vee b')\wedge b\big)=b'\vee\big((a\odot_S b)\wedge b\big)\leq \\
		& \leq b'\vee(c\wedge b)=b\rightarrow_S c.
	\end{align*}
	Similarly,  $a\leq b\rightarrow_S c$ implies
	\begin{align*}
		a\odot_S b&=b\wedge(a\vee b')\leq b\wedge\big((b\rightarrow_S c)\vee b'\big)=b\wedge\big(b'\vee(c\wedge b)\vee b'\big)\\ 
		&=b\wedge\big((c\wedge b)\vee b'\big)\stackrel{\mbox{\rm(OI${}_{\wedge}$)}}{=}c\wedge b\leq c.
	\end{align*}
(ii) $\Rightarrow$ (i): \\
Let $a,b\in L$. Since  $	a\odot_S b' \leq a\odot_S b'$ we obtain 
	\begin{align*}
		a\leq 	b' \rightarrow (a\odot b')=b\vee \big(b'\wedge [b'\wedge (a\vee b)]\big)%
		=b\vee \big(b'\wedge (a\vee b)\big).
	\end{align*}
	
	Hence $a\vee b\leq b\vee \big(b'\wedge (a\vee b)\big)$. Because the converse inequality is 
	evident we have 
	$$
	a\vee b= b\vee \big(b'\wedge (a\vee b)\big). 
	$$
	
	Similarly, since $a\rightarrow_S b \leq a\rightarrow_S b$ we obtain
	\begin{align*}
		b\geq (a\rightarrow_S b)\odot_S a=a\wedge \big(a'\vee [a'\vee (a\wedge b)]\big)%
		=a\wedge \big(a'\vee (a\wedge b)\big).
	\end{align*}
	Clearly, $a\wedge \big(a'\vee (a\wedge b)\big)\leq a\wedge b$ and the converse 
	inequality is evident. We conclude that 
	$$
	a\wedge \big(a'\vee (a\wedge b)\big)= a\wedge b. 
	$$\end{proof}

For the case of lattices, we can state and prove the following result showing that condition (A) 
for the connectives $\odot_{S}$  and $ \rightarrow_{I} $ already yields orthomodularity.

\begin{theorem}\label{th3}
Let $\mathbf L=(L,\vee,\wedge,{}',0,1)$ be a bounded lattice with an antitone involution 
and assume {\rm(A)} for 
$\odot_{S}$  and $ \rightarrow_{I} $. Then $\mathbf L$ is orthomodular.
\end{theorem}

\begin{proof}
For $a,b\in L$ every of the following statements implies the next one:
\begin{align*}
a\odot_S b' & \leq a\odot_S b', \\
       a & \leq b'\rightarrow_I(a\odot_S b'), \\
       a & \leq\big(b'\wedge(a\vee b)\big)\vee\Big(b\wedge\big(b\vee(a'\wedge b')\big)\Big), \\
       a & \leq\big(b'\wedge(a\vee b)\big)\vee b, \\
a\vee b & \leq\big(b'\wedge(a\vee b)\big)\vee b.
\end{align*}
The converse inequality is evident. Thus $a\vee b  =\big(b'\wedge(a\vee b)\big)\vee b$.
\end{proof}

Our next goal is to obtain a result analogous to Theorem~\ref{omidentity} but formulated for posets. Here we work with subsets of a given poset $P$ instead of its elements, thus we should use the equality sign ``$\approx_2$'' instead of ``$=$''.

\begin{theorem}\label{OMUI}
Let $\mathbf P=(P,\leq,{}',0,1)$ be an orthogonal mub-complete poset. Then the following conditions are equivalent:
\begin{enumerate}[{\rm(i)}]
\item $\mathbf P$ is an orthomodular poset.
\item $\mathbf P$ satisfies the identity 
\begin{enumerate}
\item[{\rm(OM$_U$)}] $x\vee\big(\Min U(x,y)\wedge x'\big)\approx\Min U(x,y)$.
\end{enumerate}
\item $\mathbf P$ satisfies the condition
\begin{enumerate}
\item[{\rm(OM$_{UE}$)}] $x\vee\big(\Min U(x,y)\wedge x'\big)\approx_2\Min U(x,y)$.
\end{enumerate}
\end{enumerate}	
\end{theorem}

\begin{proof}
(i) $\Rightarrow$ (ii): \\
This follows immediately from $x\vee(z\wedge x')=z$ for all $z\in\Min U(x,y)$.\\
(ii) $\Rightarrow$ (iii): \\
It is transparent. \\
(iii) $\Rightarrow$ (i): \\
Assume that $x\leq y$. Then $\Min U(x,y)=\{y\}$ 
	and $x\vee\big(\Min U(x,y)\wedge x'\big)=\{x\vee(y\wedge x')\}$. Hence 
	$x\vee(y\wedge x')\leq y\leq x\vee(y\wedge x')$. We conclude that 
	$y= x\vee(y\wedge x')$.
\end{proof} 	

Now we are ready to formulate our result for posets. Let us note that the operation $\odot$ defined for lattices is now replaced by a binary operator.

\begin{lemma}\label{AisB}
	Let $\mathbf P=(P,\leq,{}',0,1)$ be an orthogonal mlb-complete poset. 
	 Then {\rm(A)} and {\rm(B)}, and  {\rm(A)}${}_{21}$ 
	and  {\rm(B)}${}_{12}$  
	for the connectives $\odot_{S}$  and $ \rightarrow_{S} $ are equivalent, respectively.
\end{lemma}

\begin{proof}  We only show that {\rm(A)} and {\rm(B)} are equivalent. The remaining equivalence 
	follows the same proof steps.

Assume {\rm(A)}. Let $a,b,c\in P$. 
	Suppose that $a\leq b\rightarrow_S c$. Then 
	$a\leq b'\vee\Max L(c,b)$. Put $x=c'$, $y=b$, and $z=a'$. Hence 
	$z'\leq y'\vee\Max L(x',y)$. We obtain 
	$x\odot_S y= y\wedge\Min U(x,y')\leq z$. By {\rm(A)} we have 
	$x\leq y\rightarrow_S z=y'\vee\Max L(z,y)$, i.e., $c'\leq b'\vee \Max L(a',b)$. 
	We conclude that $a\odot_S b=b\wedge\Min U(a,b')\leq c $ and {\rm(B)} holds. 
	
	Following the same reasoning we obtain the converse, i.e., {\rm(B)} 
	implies {\rm(A)}.
\end{proof}

As for lattices in Theorem \ref{omidentity}, we obtain, for orthogonal mlb-complete posets the following theorem:

\begin{theorem}\label{SASOM}
Let $\mathbf P=(P,\leq,{}',0,1)$ be an orthogonal mlb-complete poset. Then the following are equivalent:
	\begin{enumerate}[{\rm(i)}]
		\item $\mathbf P$ is orthomodular.
			\item $x\odot_S y\leq_2 z$ if and only if $x\leq_1 y\rightarrow_S z$. 
	\end{enumerate}
\end{theorem}

\begin{proof}
(i) $\Rightarrow$ (ii): \\
Let $a,b,c\in P$. 
	Suppose that $a\leq_1 b\rightarrow_S c$. Then 
	$a\leq_1  b'\vee\Max L(c,b)$. There is $z\in \Max L(c,b)$ such that 
	$a\leq  b'\vee z\in U(a,b')$. 
	We compute:
	$$
	a\odot_S b=b\wedge\Min U(a,b')\leq_2 b\wedge (b'\vee z)=z\leq c.
	$$
	We conclude that (B)$_{12}$ and by Lemma \ref{AisB} also  (A)$_{21}$ hold 
	 for $\odot_S$ and $\rightarrow_S$. Hence we conclude (ii). \\
(ii) $\Rightarrow$ (i): \\
Assume that $a\leq b$, $a,b\in P$. Let us compute:
	$$
	b\odot_Sa'=a'\wedge \Min U(a,b)= \{a'\wedge b\}\leq_{2} b.
	$$
	Hence by (A)$_{21}$ we obtain $b\leq_1 a'\rightarrow_S b= \{a\vee (a'\wedge b)\}\leq b$, i.e., 
	$b= a\vee (a'\wedge b)$.
\end{proof}

The next result is analogous to Theorem \ref{th3} but formulated for posets.

\begin{theorem}\label{posth3}
Let $\mathbf P=(P,\leq,{}',0,1)$ be an orthogonal poset	and assume {\rm(A)}$_{21}$  for $\odot_{S}$ and $\rightarrow_{I}$. Then $\mathbf P$ is orthomodular.
\end{theorem}

\begin{proof}
Let $a,b\in L$ and $a\leq b$. Then $b\odot_S a'=a'\wedge(b\vee a)=a'\wedge b$ and
 $\{b\odot_S a'\}\leq_2 \{b\odot_S a'\}$. Therefore 
	$b\leq_{1} a'\rightarrow_I (a'\wedge b)= (a'\wedge b)\vee \Max L(a, a\vee b')=%
	\{a\vee (a'\wedge b)\}\leq b$. We conclude that $a\vee (a'\wedge b)= b$.
\end{proof}

\begin{corollary}
\begin{enumerate}[{\rm(i)}]
\item If\/ $\mathbf L=(L,\vee,\wedge,{}',0,1)$ is a paraorthomodular lattice  
and if\/ $\odot_S$ and $\rightarrow_I$ satisfy {\rm(A)} then $\mathbf L$ is orthomodular.
\item If\/ $\mathbf P=(P,\leq,{}',0,1)$ is a sharply paraorthomodular poset and if\/ $\odot_S$ and $\rightarrow_I$ satisfy {\rm(A)}$_{21}$  then $\mathbf P$ is orthomodular.
\end{enumerate}
\end{corollary}

\begin{proposition}\label{adji}
Let $\mathbf P=(P,\leq,{}',0,1)$ be an orthogonal mlb-complete poset
	and assume  that there exists a binary operator $\odot_I$ on $P$ such that 
	$x\odot_I y\leq z$ if and only if 
	$x\leq y\rightarrow_I z$ for all $x, y, z\in P$.
	Then 
	\begin{enumerate}[{\rm(i)}]
		\item $x\odot_I x'=\{0\}$ for all $x \in P$, 
		\item $\mathbf P$ is an orthomodular poset, 
		\item $x\odot_I y\leq_1  \Max L(x,y)\leq x, y$ for all $x, y\in P$,
		\item $\mathbf P$ is weakly Boolean.
	\end{enumerate}
\end{proposition}

\begin{proof}
\begin{enumerate}[(i)]
\item Let $x\in P$. Then 
	$x\leq \{x\}= 0\vee \Max L(x,1)=x'\rightarrow_I 0$. Hence also 
	$x\odot_I x'\leq \{0\}$, i.e., $x\odot_I x'=\{0\}$. 
\item Let $x, y\in P$, $x\leq y$. From (i) we know that 
	$y\odot_I y'=\{0\}\leq x$. Therefore 
	$y\leq y'\rightarrow_I x=x\vee \Max L(y, x')=\{x\vee (y\wedge x')\}$. 
	Hence 	$y\leq x\vee (y\wedge x')\leq y$, i.e., 
	$y= x\vee (y\wedge x')$. 
\item  Let $x, y\in P$. By  (ii), Theorem \ref{th1}, (i) and (iii),  and 
the fact that every orthomodular poset is complemented we know that 
		$	x\leq y\rightarrow_I x $ and 	$	x\leq y\rightarrow_I  y=\{y\vee y'\}=\{1\}$, respectively. Hence 
		$	x\odot_I y\leq x$ and $	x\odot_I y\leq y$.  
		We conclude $x\odot_I y\leq_1  \Max L(x,y)$. 	
\item Let $x, y\in P$ such that $x\wedge y=0$ and $x\wedge y'=0$. From (iii) we conclude 
		$x\odot_I y=\{0\}$ and $x\odot_I y'=\{0\}$. By adjointness and 
		Theorem \ref{th1} (iii) we have $x\leq y\rightarrow_I 0=\{y'\}$ and 
		$x\leq y'\rightarrow_I 0=\{y\}$. Since $\mathbf P$  is complemented we obtain $x\leq y\wedge y'=0$. 
\end{enumerate}
\end{proof}
	
\begin{proposition}\label{adjibp}
		Let $\mathbf P=(P,\leq,{}',0,1)$ be an orthogonal 
		 Boolean poset with maximality property.  Then  $\mathbf P$ is a Boolean 
		 algebra.
\end{proposition}

\begin{proof}  Recall that $\mathbf P$ is orthomodular. Namely, let $x, y\in P$, 
	$x\leq y$. Since $\mathbf P$ is a distributive poset we can compute:
	\begin{align*}
		LU\big(x, L(x',y)\big)&= L\big(U(x,x'),U(x,y)\big)=%
		L\big(\{1\}, U(y)\big)=LU(y).
	\end{align*}

Let us show that  $\mathbf P$ is a weakly Boolean poset. 
Suppose $x, y\in P$ such that $x\wedge y=0$ and $x\wedge y'=0$. 
Again by distributivity we obtain 
\begin{align*}
	L(x)=L\big(x, U(y,y')\big)&= LU\big(L(x,y),L(x,y')\big)=%
	LU\big(\{0\}, \{0\}\big)=L(0).
\end{align*}
	From \cite[Theorem 4.2]{tkadlec} we conclude that $\mathbf P$ is a Boolean algebra.
\end{proof}	

Finally, we are going to show that $\odot_I$ and $\rightarrow_I$ can form an 
adjoint pair in orthogonal 
mlb-complete posets only in a very specific case, i.e., if the poset in question is a Boolean algebra.

\begin{theorem}\label{adjebp}
	Let $\mathbf P=(P,\leq,{}',0,1)$ be an orthogonal 
	mlb-complete poset. Then  the following conditions are equivalent:
	\begin{enumerate}[{\rm(i)}]
		\item There exists a binary operator  $\odot_I$ on $\mathbf P$
		such that 
		\begin{center}
		$x\odot_I y\leq z$ if and only if 
		$x\leq y\rightarrow_I z$
		\end{center} for all $x, y, z\in P$.
		\item $\mathbf P$ is a Boolean algebra. 
	\end{enumerate}
\end{theorem}

\begin{proof}
(i) $\Rightarrow$ (ii): \\
From Proposition \ref{adji} we know that $\mathbf P$ is a weakly Boolean poset which has the maximality 
	property. 
	From \cite[Theorem 4.2]{tkadlec} we obtain that $\mathbf P$ is a Boolean algebra. \\
(ii) $\Rightarrow$ (i): \\
Clearly, $x\rightarrow_I y=y\vee (x'\wedge y')=(y\vee x')\wedge (y\vee y')%
=y\vee x'=(x'\vee x)\wedge (y\vee x')=x'\vee (x\wedge y)=x\rightarrow_S y$ 
and $x\odot_S y=y\wedge (x\vee y')=(y\wedge x)\vee (y\wedge y')=x\wedge y$. 

Since $\mathbf P$ is an orthomodular lattice from Theorem \ref{omidentity} we immediately obtain that  
\begin{center}
	$x\odot_S y\leq z$ if and only if 
	$x\leq y\rightarrow_S z$  if and only if 
	$x\leq y\rightarrow_I z$
\end{center} for all $x, y, z\in P$. Hence $\odot_I$ exists, and  $\odot_I=\odot_S=\wedge$ and $x\rightarrow_I y=x\rightarrow_S y=y\vee x'$ realize the required adjointness. 
\end{proof} 

\begin{corollary}
	Let   $\mathbf L=(L,\vee,\wedge,{}',0,1)$ be a bounded lattice with an antitone involution. 
	Then the following conditions are equivalent:
	\begin{enumerate}[{\rm(i)}]
		\item There exists a binary operation  $\odot_I$ 
		such that $x\odot_I y\leq z$ if and only if 
		$x\leq y\rightarrow_I z$ for all $x, y, z\in L$.
		\item $\mathbf L$ is a Boolean algebra. 
	\end{enumerate}
\end{corollary}

\subsection*{Acknowledgements}
Ivan Chajda and Helmut L\"anger gratefully acknowledge the Austrian Science Fund (FWF), project I~4579-N, and the Czech Science Foundation (GA\v CR), project 20-09869L, entitled ``The many facets of orthomodularity''. Antonio Ledda gratefully acknowledges the support of MIUR within the project PRIN 2017:  “Logic and cognition. Theory, experiments, and applications”, CUP: 2013YP4N3.

Authors' addresses:

Ivan Chajda \\
Palack\'y University Olomouc \\
Faculty of Science \\
Department of Algebra and Geometry \\
17.\ listopadu 12 \\
771 46 Olomouc \\
Czech Republic \\
ivan.chajda@upol.cz

Davide Fazio \\
Universit\`a degli Studi di Teramo \\
Facolt\`a di Science della Comunicazione \\
Campus ``Aurelio Saliceti'' \\
Via R.\ Balzarini 1 \\
64100 Teramo \\
Italy \\
dfazio2@unite.it

Helmut L\"anger \\
TU Wien \\
Faculty of Mathematics and Geoinformation \\
Institute of Discrete Mathematics and Geometry \\
Wiedner Hauptstra\ss e 8-10 \\
1040 Vienna \\
Austria, and \\
Palack\'y University Olomouc \\
Faculty of Science \\
Department of Algebra and Geometry \\
17.\ listopadu 12 \\
771 46 Olomouc \\
Czech Republic \\
helmut.laenger@tuwien.ac.at

Antonio Ledda \\
University of Cagliari \\
A.\ Lo.\ P.\ Hi.\ S Research Group \\
Via Is Mirrionis 1 \\
09123 Cagliari \\
Italy \\
antonio.ledda@unica.it

Jan Paseka \\
Masaryk University \\
Faculty of Science \\
Department of Mathematics and Statistics \\
Kotl\'a\v rsk\'a 2 \\
611 37 Brno \\
Czech Republic \\
paseka@math.muni.cz

\begin{thebibliography}{99}
\bibitem{abramsky}
S.~Abramsky and A. Jung, Domain theory. In Handbook of Logic in Computer Science, 
Vol. III, Oxford Univ.\ Press, 1994, 1--168.
\bibitem{B}
L.~Beran, Orthomodular Lattices. Algebraic Approach. Reidel, Dordrecht 1985. ISBN 90-277-1715-X.
\bibitem{BirkhoffvonNeumann}
G.~Birkhoff and J.~von~Neumann, {The logic of quantum mechanics}. {Annals of Mathematics}, {\bf37} (4), (1936), 823--843.
\bibitem{CF}
I.~Chajda and D.~Fazio, On residuation in paraorthomodular lattices. Soft Computing {\bf24} (2020), 10295--10304.
\bibitem{CFLLP}
I.~Chajda, D.~Fazio, H.~L\"anger, A.~Ledda and J.~Paseka, Algebraic properties of paraorthomodular posets. Logic J.\ IGPL {\bf30} (2022), 840--869.
\bibitem{CL21a}
I.~Chajda and H.~L\"anger, How to introduce the connective implication in orthomodular posets. Asian-Eur.\ J.\ Math.\ {\bf14} (2021), 2150066-1 -- 2150066-8.
\bibitem{CL21b}
I.~Chajda and H.~L\"anger, Residuation in finite posets. Math.\ Slovaca {\bf71} (2021), 807--820.
\bibitem{CL22}
I.~Chajda and H.~L\"anger, The logic of orthomodular posets of finite height.	Log.\ J.\ IGPL {\bf30} (2022), 143--154.
\bibitem{Chovanec}
F.~Chovanec and M.~Jure\v{c}kov\'a, MV-algebra pasting. \emph{International Journal of Theoretical Physics}. {\bf42}  (2003), 1913--1926.
\bibitem{Czela1975a}
J.~Czelakowski, Logics based on partial Boolean $\sigma$-algebras. I. {Studia Logica}, {\bf 33} (4), (1974), 371--396.
\bibitem{Czela1975}
J.~Czelakowski, Logics based on partial Boolean $\sigma$-algebras. II. {Studia Logica}, {\bf 34} (1), (1975), 69--86.
\bibitem{Czela1979}
J.~Czelakowski, Partial Boolean algebras in a broader sense. {Studia Logica}, {\bf 38} (1), (1979), 1--16.
\bibitem{FLP}
D.~Fazio, A.~Ledda and F.~Paoli, On Finch's conditions for the completion of orthomodular posets. Found.\ Sci.\ (2020).
\bibitem F
P.~D.~Finch, On orthomodular posets. J.\ Austral.\ Math.\ Soc.\ {\bf11} (1970), 57--62.
\bibitem{GLP16}
R.~Giuntini, A.~Ledda and F.~Paoli, A new view of effects in a Hilbert space. Studia Logica {\bf104} (2016), 1145--1177.
\bibitem{GLP17}
R.~Giuntini, A.~Ledda and F.~Paoli, On some properties of PBZ*-lattices. Internat.\ J.\ Theoret.\ Phys.\ {\bf56} (2017), 3895--3911.
\bibitem{Greechie1}
R.~J. Greeehie, Orthomodular lattices admitting no states, Journ.\ Comb.\ Theory {\bf10} (1971), 119--132.
\bibitem K
G.~Kalmbach, Orthomodular Lattices, Academic Press, London 1983. ISBN 0-12-394580-1.
\bibitem{KochenSpecker2}
S.~Kochen, and E.P.~Specker, Logical structures arising in quantum theory, in J.W.~Addison, L.~Henkin and A.~Tarski (eds.), \emph{The Theory of Models}, in Studies in Logic and the Foundations of Mathematics, North-Holland, Amsterdam 1965, 177--189.
\bibitem{Pulman}
S,~Pulmannov\'a, Compatibility and partial compatibility
in quantum logics, Annales de l’I. H. P., section A, {\bf 34} (1981), 391--403.
\bibitem{Suppes}
P. Suppes, Logics appropriate to empirical theories, in \emph{Symposium on the Theory of Models}, Proceedings of the 1963 International Symposium at Berkeley, North-Holland, Amsterdam, 1965, 364--375
\bibitem{tkadlec}
J.~Tkadlec, Conditions that force an orthomodular poset to be a Boolean algebra. Tatra Mountains Mathematical Publications {\bf10} (1997), 55--62.
\bibitem{xie}
Y.~Xie, Y.~Li and A.~Yang, The pasting constructions of lattice ordered effect algebras, \emph{Information Sciences,} {\bf180}(12) (2010), 2476--2486.
\end{thebibliography}
\end{document}